\documentclass[12pt,a4paper]{amsart}
\usepackage{amsmath,amssymb}
\usepackage{graphicx}
\usepackage{calligra}
\usepackage{amsfonts}
\newtheorem{Thm}{Theorem}
\newtheorem*{Thm*1}{Theorem \ref{1}}
\newtheorem*{Thm*2}{Theorem \ref{2}}
\newtheorem*{Thm*mu}{Theorem \ref{mu}}
\newtheorem*{Prop*UB}{Proposition \ref{UB}}
\newtheorem{Prop}{Proposition}

\newtheorem*{Conj}{Conjecture}

\newtheorem{Rmk}{Remark}
\newtheorem*{Rmk*}{Remark}

\begin{document}
\title[Volume bounds of the Ricci flow on closed manifolds]
{Volume bounds of the Ricci flow on closed manifolds}
\author{Chih-Wei Chen}
\address{(Chih-Wei Chen) National Kaohsiung Normal University, Taiwan}
\email{BabbageTW@gmail.com; chencw@nknu.edu.tw}
\author{Zhenlei Zhang}
\address{(Zhenlei Zhang) Department of Mathematics, Capital Normal University, Beijing, China}
\email{Zhleigo@aliyun.com}
\date{March, 2018}
\keywords{Ricci flow, volume estimate, $\mu$-entropy}

\begin{abstract}
Let $\{g(t)\}_{t\in [0,T)}$ be the solution of the Ricci flow on a closed Riemannian manifold $M^n$ with $n\geq 3$.
Without any assumption, we derive lower volume bounds of the form ${\rm Vol}_{g(t)}\geq C (T-t)^{\frac{n}{2}}$, 
where $C$ depends only on $n$, $T$ and $g(0)$. In particular, we show that
$${\rm Vol}_{g(t)} \geq e^{ T\lambda-\frac{n}{2}} \left(\frac{4}{(A(\lambda-r)+4B)T}\right)^{\frac{n}{2}}\left(T-t\right)^{\frac{n}{2}},$$
where $r:=\inf_{\|\phi\|_2^2=1} \int_M R\phi^2 \ d{\rm vol}_{g(0)}$,  
$\lambda:=\inf_{\|\phi\|_2^2=1} \int_M  4|\nabla\phi|^2+R\phi^2\ d{\rm vol}_{g(0)}$ and  $A,B$ are Sobolev constants of $(M,g(0))$. This estimate is sharp in the sense that it is achieved by the unit sphere with scalar curvature $R_{g(0)}=n(n-1)$ and $A=\frac{4}{n(n-2)}\omega_n^{-\frac{2}{n}}$, $B=\frac{n-1}{n-2}\omega_n^{-\frac{2}{n}}$.

On the other hand, if the diameter satisfies ${\rm diam}_{g(t)}\leq c_1\sqrt{T-t}$ and there exist a point $x_0\in M$ such that $R(x_0,t)\leq c_2(T-t)^{-1}$, then we have ${\rm Vol}_{g(t)}\leq C (T-t)^{\frac{n}{2}}$ for all $t>\frac{T}{2}$, where $C$ depends only on $c_1,c_2,n,T$ and $g(0)$. 

\end{abstract}

\maketitle

\section{Introduction}
Let $(M^n,g)$ be a closed Riemannian manifold with dimension $n\geq 3$
and $A$ and $B$ be any Sobolev constants of $(M^n,g)$, i.e.,
$$ \left( \int_M |u|^{\frac{2n}{n-2}}\ d{\rm vol} \right)^{\frac{n-2}{n}}\leq A\int_M  |\nabla u|^2\ d{\rm vol}+B\int_M u^2\ d{\rm vol}$$
for all $u\in W^{1,2}(M)$.
In \cite{ZZhang07}, one of us observed that the Ricci flow on a closed manifold has a volume lower bound in terms of Sobolev constants. Especially, when $\lambda:=\inf_{\|\phi\|_2^2=1} \int_M  4|\nabla\phi|^2+R\phi^2\ d{\rm vol}_{g(0)}$ is non-positive, one obtains ${\rm Vol}_{g(t)} \geq C_1e^{-C_2t}$ for all $t>0$, where $C_1$ and $C_2$ depend only on $n$ and $g(0)$.
It means that the manifold cannot extinct at finite time and every blow up limit must be non-compact.


On the other hand, for positive $\lambda$, R. Ye \cite{Ye07} derived the following volume lower bound by using the estimate of $A(t)$ and $B(t)$: 
\begin{Prop}[R. Ye]
Assume that $\lambda$ is positive. 
Then we have for any time $t\in[0,T)$ 
$${\rm Vol}_{g(t)} \geq e^{-\frac{1}{4}-C}\ \mbox{ when }\ \bar R(t) \leq 0$$
and 
$${\rm Vol}_{g(t)} \geq e^{-\frac{1}{4}-C} \bar R(t)^{-\frac{n}{2}}\ \mbox{  when }\ \bar R(t) >0,$$
where $\bar R(t):=-\hspace{-4mm}\int R \ d{\rm vol}_{g(t)}$ and $C$ depends on $n,g_0,A,B,\lambda,{\rm Vol}_{g(0)}$ and $\max R_{g(0)}^-$.
\end{Prop}
Note that the constants $C$'s in Zhang's estimate depend only on the initial metric $g(0)$, while the constants in Ye's estimate, namely for the case $\lambda>0$, might depend on $\bar R(t)$.

In this article, we find a unified way to derive several volume bounds, whose 
proofs do not rely on the definite sign of $\lambda$.

\begin{Thm}\label{1}
Let $\{g(t)\}_{t\in [0,T)}$ be the solution of the Ricci flow on a closed Riemannian manifold $M^n$ with $n\geq 3$
and $A,B$ be Sobolev constants of $(M,g(0))$.
Then $${\rm Vol}_{g(t)} \geq e^{ T\lambda - \frac{na}{8} \left( A(\lambda-r ) +4 B\right)} a^{\frac{n}{2}}\left(\frac{T-t}{T}\right)^{\frac{n}{2}}$$
for all $a\in(0,\frac{8T}{nA}]$, where 
$r:=\inf_{\|\phi\|_2^2=1} \int_M R\phi^2 \ d{\rm vol}_{g(0)}$ and
$\lambda:=\inf_{\|\phi\|_2^2=1} \int_M  4|\nabla\phi|^2+R\phi^2\ d{\rm vol}_{g(0)}$.
In particular, when $a=\frac{8T}{nA}$, we obtain a lower bound 
$${\rm Vol}_{g(t)}  \geq e^{T(r-4BA^{-1})}\left(\frac{8}{nA}\right)^{\frac{n}{2}}( T-t )^{\frac{n}{2}},$$
which does not depend on $\lambda$.
\end{Thm}

When choosing $B\geq \frac{nA}{8T}$ and $a=4(A(\lambda-r)+4B)^{-1}$, we have the following theorem
which shows that our estimate is sharp.

\begin{Thm}\label{2}
Let $\{g(t)\}_{t\in [0,T)}$ be the solution of the Ricci flow on a closed Riemannian manifold $M^n$ with $n\geq 3$.
Suppose that $A$ and $B\geq \frac{nA}{8T}$ are Sobolev constants of $(M,g(0))$. 
Then 
$${\rm Vol}_{g(t)} \geq e^{ T\lambda-\frac{n}{2}}\left(\frac{4}{A(\lambda-r)+4B}\right)^{\frac{n}{2}}\left(\frac{T-t}{T}\right)^{\frac{n}{2}}.$$
The bound is achieved when $(M,g(0))$ is the unit sphere with $A=\frac{4}{n(n-2)}\omega_n^{-\frac{2}{n}}$ and $B=\frac{n-1}{n-2}\omega_n^{-\frac{2}{n}}$.
\end{Thm}

The proof of Theorem \ref{1} is based on the monotonicity of Perelman's $\mu$-entropy. 
Recall that $\mu$-entropy is defined by
$$\mu(g(t),\tau(t)):=\inf_{\|\phi\|_2^2=(4\pi\tau)^{\frac{n}{2}}} (4\pi\tau)^{-\frac{n}{2}}\int_M \left[   \tau(4|\nabla \phi|^2 +R\phi^2)-\phi^2\ln\phi^2 -n\phi^2  \right] d{\rm vol}_{g(t)}$$
for all $\phi(\cdot,t)\in W^{1,2}(M)$.
When fixing $t$ and choosing $\phi^2$ to be the constant $(4\pi\tau)^{\frac{n}{2}}{\rm Vol}_{g(t)}^{-1}$,
one obtains 
$$\mu(g(t),\tau(t))\leq -\hspace{-4.5mm}\int_M   \tau R\  d{\rm vol}_{g}-\ln\frac{(4\pi\tau)^{\frac{n}{2}}}{{\rm Vol}_{g(t)}}-n=-\tau \left(\ln {\rm Vol}_{g(t)}\right)' -\ln\frac{(4\pi\tau)^{\frac{n}{2}}}{{\rm Vol}_{g(t)}}-n$$
This shows that the evolution of volume is closely related to the $\mu$-entropy. 
The relationship between $\mu$ and ${\rm Vol}$ has been studied by one of the authors in \cite{ZZhang07},
especially for manifolds with $\lambda<0$.
Here we derive results for generic manifolds. 

\begin{Thm}\label{mu}
Let $\{g(t)\}_{t\in [0,T)}$ be the solution of the Ricci flow on a closed Riemannian manifold $M^n$ with $n\geq 2$. Denote $\mu=\inf_{\|\phi\|_2^2=(4\pi T)^{\frac{n}{2}}}\mathcal{W}(g(0),\phi,T)$. Then 
$$ {\rm Vol}_{g(t)}  
\geq(4\pi)^{\frac{n}{2}}e^{\mu+\frac{n}{2}}(T-t)^{\frac{n}{2}}.$$
As a consequence of the volume lower bound, for any closed Riemannian manifold $(M,g)$, one has
$$\mu(g,T) \leq -\frac{n}{2}+ \ln {\rm Vol}_{g}(M) -\frac{n}{2}\ln 4\pi T.$$ 
In particular, for any Ricci flow defined on a closed manifold, the maximal time $T$ cannot exceed $\left(4\pi e\right)^{-1}(e^{-\mu(g(0),T)}{\rm Vol}_{g(0)})^{\frac{2}{n}}$.
\end{Thm}

\vspace{4mm}

If we further assume some controls on diameter and curvature, we derive the following upper bound for volume.

\begin{Prop}\label{UB}
Let $\{g(t)\}_{t\in [0,T)}$ be the solution of the Ricci flow on a closed Riemannian manifold $(M^n,g(0))$ with $n\geq 3$. If the diameter satisfies ${\rm diam}_{g(t)}\leq c_1\sqrt{T-t}$ and there exist a point $x_0\in M$ such that $R(x_0,t)\leq c_2(T-t)^{-1}$, then we have ${\rm Vol}_{g(t)}\leq C (T-t)^{\frac{n}{2}}$ for all $t>\frac{T}{2}$, where $C$ depends only on $c_1,c_2,n,T$ and $g(0)$. 
\end{Prop}

If the curvature condition in Proposition \ref{UB} is replaced by the stronger one that $R(x,t)\leq c_2(T-t)^{-1}$ {\it for all} $x\in M$, then the theorem follows directly from Q. Zhang's Theorem 1.1 in \cite{QZhang12}.
However, by mimicking Zhang's argument carefully, one can see that the Type I curvature assumption can be reduced as $R(x_0,t)\leq c_1(T-t)^{-1}$ {\it for some} $x_0\in M$. 
For the reader's convenience, we include an outline of Zhang's proof in Section 6.

The theorem above relates to the following conjecture. 
Roughly speaking, we suspect that $R$ cannot be of Type II at every point on a manifold which shrinks to a point along the Ricci flow.

\begin{Conj}
Let $\{g(t)\}_{t\in [0,T)}$ be the solution of the Ricci flow on a closed Riemannian manifold $M^n$ with $n\geq 3$ and ${\rm diam}_{g(t)}\to 0$ as $t\to T$.
Then  $$\liminf_{t\to T} (T-t) R (x,t) < \infty$$
\end{Conj}

\ \\
{\bf Acknowledgement.} The first author appreciates Mao-Pei Tsui for suggesting 
him to compare the volume of sphere and other manifolds. He is always indebted to Shu-Cheng Chang and Huai-Dong Cao for their constant supports and discussions.

\section{Lower bounds involving Sobolev constants}

Let $\{g(x,t)\}_{t\in [0,T)}$ be the solution of the Ricci flow on a closed Riemannian manifold $M^n$ with $n\geq 3$.
Perelman's $\mathcal{W}$-functional is defined by
$$\mathcal{W}(g(x,t),\phi(x,t),\tau(t)):=  (4\pi\tau)^{-\frac{n}{2}}\int_M \left[   \tau(4|\nabla \phi|^2 +R\phi^2)-\phi^2\ln\phi^2 -n\phi^2  \right] d{\rm vol}_{g}$$
for all $\phi(\cdot,t)\in W^{1,2}(M)$. 

Perelman proved the monotonicity of $\mathcal{W}$ along the Ricci flow \cite{Perelman02}. 
Precisely, for any $\tau(t)$ and $\phi$ such that $\tau'=-1$ and $\frac{\partial}{\partial t}\phi^2=-\Delta\phi^2 +(R-\frac{n}{2\tau})\phi^2$, 
$\mathcal{W}(g(x,t),\phi(x,t),\tau(t))$ and
$\mu:=\inf_{\|\phi\|_2^2=(4\pi\tau)^{\frac{n}{2}}}\mathcal{W}(g(x,t),\phi(x,t),\tau(t))$ are nondecreasing with respect to $t$.
In particular, the value of $\mu$ at $t>0$ is greater than or equal to the value at $t=0$.
Namely, 
$$\inf_{\|\phi\|_2^2=(4\pi\tau(t))^{\frac{n}{2}}}\mathcal{W}(g(t),\phi,\tau(t))\geq \inf_{\|\phi\|_2^2=(4\pi\tau(0))^{\frac{n}{2}}}\mathcal{W}(g(0),\phi,\tau(0)),$$
and equivalently
$$\sup_{\|\phi\|_2^2=(4\pi\tau(t))^{\frac{n}{2}}}-\mathcal{W}(g(t),\phi,\tau(t))\leq \sup_{\|\phi\|_2^2=(4\pi\tau(0))^{\frac{n}{2}}}-\mathcal{W}(g(0),\phi,\tau(0)).$$

Hence 
\begin{align*}
& \sup_{\|\phi\|_2^2=(4\pi\tau(t))^{\frac{n}{2}}}(4\pi\tau(t))^{-\frac{n}{2}}\int_M \left[  \phi^2\ln\phi^2 +n\phi^2 -\tau(t)(4|\nabla \phi|^2 +R\phi^2)  \right] d{\rm vol}_{g(t)}\\
\leq &\ \sup_{\|\phi\|_2^2=(4\pi \tau(0))^{\frac{n}{2}}}(4\pi \tau(0))^{-\frac{n}{2}}\int_M \left[  \phi^2\ln\phi^2 +n\phi^2 -\tau(0)(4|\nabla \phi|^2 +R\phi^2) \right] d{\rm vol}_{g(0)}.
\end{align*}
From now on, we denote $\tau=\tau(t)$, $\tau_0=\tau(0)$ and $V(t)={\rm Vol}_{g(t)}$. 
Considering the (spatial) constant function $\phi^2=(4\pi\tau)^{\frac{n}{2}}V(t)^{-1}$ at time $t$, one derives 
\begin{align}\label{W}
& -\hspace{-4.75mm}\int_M \left[   \ln \frac{(4\pi \tau)^{\frac{n}{2}}}{V(t)}+n -\tau R \right] d{\rm vol}_{g(t)} \notag\\
\leq &\ \sup_{\|\phi\|_2^2=(4\pi \tau_0)^{\frac{n}{2}}}(4\pi \tau_0)^{-\frac{n}{2}}\int_M \left[  \phi^2\ln\phi^2 +n\phi^2 -\tau_0(4|\nabla \phi|^2 +R\phi^2) \right] d{\rm vol}_{g(0)}\\
=&\  \sup_{\|\phi\|_2^2=1}\int_M \left[  \phi^2(\ln\phi^2 +\ln (4\pi \tau_0)^{\frac{n}{2}}) +n\phi^2 -\tau_0(4|\nabla \phi|^2 +R\phi^2) \right] d{\rm vol}_{g(0)},\notag
\end{align}
i.e.,
\begin{align*}
& -\ln V(t)+\ln(4\pi\tau)^{\frac{n}{2}} -\tau -\hspace{-4.75mm}\int_MR\  d{\rm vol}_{g(t)}\\
\leq &\  \sup_{\|\phi\|_2^2=1}\int_M \left[  \phi^2\ln\phi^2   -\tau_0(4|\nabla \phi|^2 +R\phi^2) \right] d{\rm vol}_{g(0)} +\ln (4\pi \tau_0)^{\frac{n}{2}}.
\end{align*}
On the other hand, since $\tau'=-1$ and $\frac{d}{dt} V(t)=-\int_MR\ d{\rm vol}$, one has
$$\frac{d}{dt} \left(\tau\ln V(t) \right)=-\ln V(t)-\tau-\hspace{-5mm} \int_M R \ d{\rm vol}.
$$
Therefore,
\begin{align}\label{key}
& \frac{d}{dt}\left(\tau\ln V(t) \right)\leq   \sup_{\|\phi\|_2^2=1}\int_M \left[  \phi^2\ln\phi^2  -\tau_0(4|\nabla \phi|^2 +R\phi^2) \right] d{\rm vol}_{g(0)} - \ln\left(\frac{\tau}{\tau_0}\right)^{\frac{n}{2}}.
\end{align}

Inequalities (\ref{W}) and the right hand side of (\ref{key}) were observed and used before, see for example \cite{QZhang07, Ye07}. The left hand side of (\ref{key}) also occurred in a more general form in \cite{QZhang07,QZhang12}.  
However, they were used for tracing the evolution of Sobolev constants or fundamental solutions of the (conjugate) heat equation, instead of the global volume function $V(t)$. 
In Ye's Proposition (cf. Proposition 1 in the introduction), Ye needs the positivity assumption 
to make sure that $A(t)$ and $B(t)$ are under control along the Ricci flow.
Indeed, as pointed out by Ye, when the assumption $\lambda>0$ is removed, Hamilton-Isenberg's example shows that local volume could collapse, which means that $A(t)$ and $B(t)$ must become wild. So Ye guessed that the positivity assumption of $\lambda$ is indispensable \cite[p. 4]{Ye07}. 
However, we show that, although the Sobolev constants and the local volume could be bad, the global volume remains under control.

\begin{Thm*1}
Let $\{g(t)\}_{t\in [0,T)}$ be the solution of the Ricci flow on a closed Riemannian manifold $M^n$ with $n\geq 3$
and $A,B$ be Sobolev constants of $(M,g(0))$.
Then $${\rm Vol}_{g(t)} \geq e^{ T\lambda - \frac{na}{8} \left( A(\lambda-r ) +4 B\right)} a^{\frac{n}{2}}\left(\frac{T-t}{T}\right)^{\frac{n}{2}}$$
for all $a\in(0,\frac{8T}{nA}]$, where 
$r:=\inf_{\|\phi\|_2^2=1} \int_M R\phi^2 \ d{\rm vol}_{g(0)}$ and
$\lambda:=\inf_{\|\phi\|_2^2=1} \int_M  4|\nabla\phi|^2+R\phi^2\ d{\rm vol}_{g(0)}$.
In particular, when $a=\frac{8T}{nA}$, we obtain a lower bound 
$${\rm Vol}_{g(t)}  \geq e^{T(r-4BA^{-1})}\left(\frac{8}{nA}\right)^{\frac{n}{2}}( T-t )^{\frac{n}{2}},$$
which does not depend on $\lambda$.
\end{Thm*1}

\begin{proof}

For all $\phi\in W^{1,2}(M)$, the Sobolev inequality implies that 
\begin{align*}
 \int_M \phi^2\ln\phi^2\ d{\rm vol}_{g(0)}
\leq &\ \frac{n}{2}\ln\left( A\int_M |\nabla\phi|^2\ d{\rm vol}_{g(0)}+B \int_M\phi^2 d{\rm vol}_{g(0)}  \right)\\
\leq &\ \frac{n}{2}a \left(A\int_M |\nabla\phi|^2\ d{\rm vol}_{g(0)}+B \int_M\phi^2 d{\rm vol}_{g(0)} \right)   -\frac{n}{2}\ln a -\frac{n}{2}\\
=&\ \frac{naA}{8} \int_M 4|\nabla\phi|^2\ d{\rm vol}_{g(0)}+\frac{naB}{2}\int_M\phi^2\ d{\rm vol}_{g(0)} -\frac{n}{2}\ln a -\frac{n}{2}.
\end{align*}
Note that the second inequality follows from the fact $\ln x \leq ax -\ln a-1$.
Hence
\begin{align*}
& \ \int_M \left[  \phi^2\ln\phi^2  -\tau_0(4|\nabla \phi|^2 +R\phi^2) \right] d{\rm vol}_{g(0)}\\
\leq &\ \left(\frac{naA}{8}-\tau_0\right)\int_M 4|\nabla\phi|^2 +R\phi^2\ d{\rm vol}_{g(0)}-\frac{naA}{8}   \int_M R\phi^2 \ d{\rm vol}_{g(0)} \\
&\  +\frac{naB}{2}\int_M\phi^2\ d{\rm vol}_{g(0)}  -\frac{n}{2}\ln a -\frac{n}{2}.
\end{align*}
Since $a\leq\frac{8\tau_0}{nA}$, $\frac{naA}{8}-\tau_0$ is nonpositive and we have
\begin{align*}
& \ \sup_{\|\phi\|_2^2=1} \int_M \left[  \phi^2\ln\phi^2  -\tau_0(4|\nabla \phi|^2 +R\phi^2) \right] d{\rm vol}_{g(0)}\\
\leq &\ \left(\frac{naA}{8}-\tau_0\right) \lambda -\frac{naA}{8}  r   +\frac{naB}{2}  -\frac{n}{2}\ln a -\frac{n}{2}\\
= &\   -\tau_0\lambda + \frac{naA}{8}\left(\lambda-r\right)   +\frac{naB}{2}  -\frac{n}{2}\ln a -\frac{n}{2},
\end{align*}
where $\lambda:=\inf_{\|\phi\|_2^2=1} \int_M  4|\nabla\phi|^2+R\phi^2\ d{\rm vol}_{g(0)}$ 
and $r:=\inf_{\|\phi\|_2^2=1} \int_M R\phi^2 \ d{\rm vol}_{g(0)}$.

Applying it to the key inequality (\ref{key}), we obtain
\begin{align*}
 \frac{d}{dt} \left(\tau\ln V(t)\right) 
\leq &\ \sup_{\|\phi\|_2^2=1}\int_M \left[  \phi^2\ln\phi^2  -\tau_0(4|\nabla \phi|^2 +R\phi^2) \right] d{\rm vol}_{g(0)} - \ln\left(\frac{\tau}{\tau_0}\right)^{\frac{n}{2}}\\
\leq &\   -\tau_0\lambda + \frac{naA}{8}\left(\lambda-r\right)      +\frac{naB}{2}  -\frac{n}{2}\ln a -\frac{n}{2}
- \ln\left(\frac{\tau}{\tau_0}\right)^{\frac{n}{2}}.
\end{align*}
Taking $\tau=T-t$ and integrating the inequality from $t$ to $T$, we have
$$ -\ln V(t) \leq - T\lambda + \frac{naA}{8}  (\lambda-r) +\frac{naB}{2}  -\ln \left(\frac{a}{T}\right)^{\frac{n}{2}} - \ln\left(T-t\right)^{\frac{n}{2}},$$
i.e.,
$$ V(t)\geq   e^{ T\lambda - \frac{naA}{8}  (\lambda-r) -\frac{naB}{2}} a^{\frac{n}{2}}\left(\frac{T-t}{T}\right)^{\frac{n}{2}}$$
for all $t\in [0,T)$. 
It is easy to see that $\lambda$ can be cancelled out when $a=\frac{8T}{nA}$.
\end{proof}


\section{Sharp volume bound and the best Sobolev constants of the sphere}

In order to acquire a sharp estimate, one has to choose a specific constant $a$ in Theorem \ref{1}. 
In fact, the lower bound in Theorem \ref{1} involves the function $f(a)=e^{-Ca}a^{\frac{n}{2}}$,
which has a unique absolute maximum at $a=\frac{n}{2C}$. So it is not hard to see that the best choice of $a$
is $4(A(\lambda-r)+4B)^{-1}$. 
However, due to a technical but crucial reason arising from the proof, $a$ cannot exceed $\frac{8T}{nA}$.
Hence the best choice of $a$ is allowed only when $B$ is chosen to be large, say $B\geq \frac{nA}{8T}$.
This fact shows that the best choice of $a$ is not necessarily given by the best Sobolev constants.

The best choice of $a$, where $f(a)$ attains it maximum, makes the lower bound in Theorem \ref{1}
becomes  $ e^{ T\lambda-\frac{n}{2}} a^{\frac{n}{2}}\left(\frac{T-t}{T}\right)^{\frac{n}{2}}$.
We shall recall some facts from the theory of Sobolev constants (cf. \cite{DruetHebey02}) and 
show that this lower bound can be attained by the shrinking sphere. 
From now on we consider closed Riemannian manifolds with dimension $n\geq 3$
and denote $\omega_n$ as the volume of the unit sphere, whose sectional curvatures are $1$.
In \cite{HebeyVaugon96}, Hebey and Vaugon showed that one can always choose
$A=\frac{4}{n(n-2)}\omega_n^{-\frac{2}{n}}$ for a given $(M^n,g)$ so that the Sobolev inequality holds.
Namely, there exists a constant $B>0$ such that
 $$ \left( \int_M |u|^{\frac{2n}{n-2}}\ d{\rm vol} \right)^{\frac{n-2}{n}}\leq \frac{4}{n(n-2)}\omega_n^{-\frac{2}{n}}\int_M  |\nabla u|^2\ d{\rm vol}+B\int_M u^2\ d{\rm vol}$$
for all $u\in W^{1,2}(M)$.
The infimum of all the $B$'s which make the inequality valid is called the best $B$-constant and is  denoted by $B_0$.
For the unit sphere, a well-known result due to T. Aubin \cite{Aubin76} states that $B_0$ is $\omega_n^{-\frac{2}{n}}$.
Hence, the Sobolev inequality holds on the unit sphere when we choose $B=\frac{n-1}{n-2}\omega_n^{-\frac{2}{n}}>B_0$
and we have

\begin{Thm*2}
Let $\{g(t)\}_{t\in [0,T)}$ be the solution of the Ricci flow on a closed Riemannian manifold $M^n$ with $n\geq 3$.
Suppose that $A$ and $B\geq \frac{nA}{8T}$ are Sobolev constants of $(M,g(0))$. 
Then 
$${\rm Vol}_{g(t)} \geq e^{ T\lambda-\frac{n}{2}}\left(\frac{4}{A(\lambda-r)+4B}\right)^{\frac{n}{2}}\left(\frac{T-t}{T}\right)^{\frac{n}{2}}.$$
The bound is achieved when $(M,g(0))$ is the unit sphere with $A=\frac{4}{n(n-2)}\omega_n^{-\frac{2}{n}}$ and $B=\frac{n-1}{n-2}\omega_n^{-\frac{2}{n}}$.
\end{Thm*2}

\begin{proof}
The first statement follows easily from Theorem \ref{1} by taking $a=4(A(\lambda-r)+4B)^{-1}$. 

For the second statement,
we consider the shrinking sphere with $R_{g(0)}=n(n-1)$, it is easy to compute that $T=\frac{1}{2(n-1)}$ and 
${\rm Vol}_{g(t)}=\left(2(n-1)\right)^{\frac{n}{2}}\omega_n\left(T-t\right)^{\frac{n}{2}}$.
On the other hand, when we choose 
$A=\frac{4}{n(n-2)}\omega_n^{-\frac{2}{n}}$ and $B=\frac{n-1}{n-2}\omega_n^{-\frac{2}{n}}$,
the lower bound becomes
$$e^{\frac{1}{2(n-1)}n(n-1)-\frac{n}{2}} \left(2(n-1)\frac{4}{A(\lambda-r)+4B}\right)^{\frac{n}{2}}\left(T-t\right)^{\frac{n}{2}}
=\left(2(n-1)\right)^{\frac{n}{2}}\omega_n\left(T-t\right)^{\frac{n}{2}}$$
because $\lambda=r=R=n(n-1)$.
\end{proof}

\section{Lower bounds involving $\mu$}

In this section, we derive lower and upper bounds of global volume in terms of $\mu(g(0),T)$.
Since we do not interpret $\mu$ by using Sobolev constants, all the results in this section hold for $n\geq 2$, instead of $n\geq 3$.

\begin{Thm*mu}
Let $\{g(t)\}_{t\in [0,T)}$ be the solution of the Ricci flow on a closed Riemannian manifold $M^n$ with $n\geq 2$. Denote $\mu=\inf_{\|\phi\|_2^2=(4\pi T)^{\frac{n}{2}}}\mathcal{W}(g(0),\phi,T)$. Then 
$$ {\rm Vol}_{g(t)}  
\geq(4\pi)^{\frac{n}{2}}e^{\mu+\frac{n}{2}}(T-t)^{\frac{n}{2}}.$$
As a consequence of the volume lower bound, for any closed Riemannian manifold $(M,g)$, one has
$$\mu(g,T) \leq -\frac{n}{2}+ \ln {\rm Vol}_{g}(M) -\frac{n}{2}\ln 4\pi T.$$ 
In particular, for any Ricci flow defined on a closed manifold, the maximal time $T$ cannot exceed $\left(4\pi e\right)^{-1}(e^{-\mu(g(0),T)}{\rm Vol}_{g(0)})^{\frac{2}{n}}$.
\end{Thm*mu}

\begin{proof}
Recall that Perelman's $\mu$-entropy
\begin{align*}
\mu(g(t),\tau):=\ & \inf_{\|\phi\|_2^2=(4\pi \tau)^{\frac{n}{2}}}\mathcal{W}(g(x,t),\phi(x,t),\tau(t))\\
=\ &  \inf_{\|\phi\|_2^2=(4\pi \tau)^{\frac{n}{2}}}(4\pi\tau)^{-\frac{n}{2}}\int_M \left[   \tau(4|\nabla \phi|^2 +R\phi^2)-\phi^2\ln\phi^2 -n\phi^2  \right] d{\rm vol}_{g}.
\end{align*}
is non-decreasing along the Ricci flow for any $\tau(t)$ and $\phi$ such that $\tau'=-1$ and $\frac{\partial}{\partial t}\phi^2=-\Delta\phi^2 +(R-\frac{n}{2\tau})\phi^2$.

Denote $V(t)={\rm Vol}_{g(t)}$, $\tau_0=\tau(0)$, $\tau=\tau(t)$, and consider $\phi^2 = (4\pi\tau)^{\frac{n}{2}}V(t)^{-1}$ at time $t$.
So, by the monotonicity of $\mu$, one has
\begin{align*}
\mu(g(0),\tau_0)
\leq\ &\mu(g(t),\tau)\\
\leq\ & (4\pi\tau)^{-\frac{n}{2}}\int_M [(\tau R\phi^2 -\phi^2\ln \phi^2 -n\phi^2)] \ d{\rm vol}_{g(t)}\\
=\ &  -\hspace{-4.75mm}\int_M [\tau R+\ln V(t)] \ d{\rm vol}_{g(t)}  -\ln(4\pi\tau)^{\frac{n}{2}} -n\\
=\ & -\tau\frac{d}{dt}(\ln V(t)) +\ln V(t) -\ln(4\pi\tau)^{\frac{n}{2}} -n
\end{align*}
and thus
$$-\frac{d}{dt}\left(\tau\ln V(t)\right)
\geq \frac{n}{2}\ln\tau +\frac{n}{2}\ln(4\pi)+n +\mu(g(0),\tau_0).$$
Taking $\tau=T-t$ and integrating the inequality from $t$ to $T$, we obtain
$$ V(t)\geq (4\pi)^{\frac{n}{2}}e^{\mu(g(0),T)+\frac{n}{2}}(T-t)^{\frac{n}{2}}.$$
\end{proof}

One may compare the upper bound of $\mu$ with a former result given by one of the authors as follows.

\begin{Prop}[\cite{ZZhang07}, cf. \cite{ChowPart3}]
For any closed Riemannian manifold $(M^n,g)$, one has 
$$\mu(g,T) \leq -n+T\lambda + e^{-1}{\rm Vol}_{g}(M) -\frac{n}{2}\ln 4\pi T.$$ 
Moreover, if $\lambda\leq 0$, then 
$$\mu(g,T) \leq -n+e^{-1} + \ln {\rm Vol}_{g}(M) -\frac{n}{2}\ln 4\pi T.$$ 
\end{Prop}

For the reader's convenience, we recall the proof.
\begin{proof}
The first inequality comes from the definition of $\mu$ and the fact $-x\ln x\leq e^{-1}$ for all $x\geq 0$.
When $\lambda<0$, we can simply remove the term $T\lambda$. 
However, a rescaling argument can do a better job. 
Indeed, since $\mu(g,T)=\mu(Qg,QT)$ for any $Q\in\mathbb{R}$,
when choosing $Q={\rm Vol}_{g}^{-\frac{2}{n}}$, one has ${\rm Vol}_{Qg}(M)=1$,
$$\mu(Qg,QT) \leq -n+QT\lambda + e^{-1}{\rm Vol}_{Qg}(M) -\frac{n}{2}\ln 4\pi QT\leq -n + e^{-1} -\frac{n}{2}\ln 4\pi QT$$
and thus the proposition is proved.
\end{proof}
More discussions about the behavior of $\mu$ and its applications can be found in \cite[Chapter 17]{ChowPart3}.

\section{Upper bounds}

Let $g(t)$, $t\in[0,T)$, be the solution of the Ricci flow on a closed Riemannian manifold $(M^n,g(0))$.
Along this flow, consider the heat kernel $G(x,t;y,s)$ for the heat operator 
$\partial_t -\Delta_x$.
Namely, fixing $y$ and $s$, $u(x,t):=G(x,t;y,s)$ satisfies 
$$\partial_t u(x,t)= \Delta_x u(x,t)\ \mbox{ and }\ \lim_{t\searrow s}u(x,t)=\delta_y(x).$$ 
One can consult Chow et al.'s book \cite[Ch. 24]{ChowPart3} for more details about the heat kernel.
In \cite[pp. 247, 251]{QZhang12}, Q. Zhang derived the following two-sided bound
for the integral heat kernel:
\begin{align*}
\left(1+C(t-s)\right)^\frac{n}{2} 
& \geq \int_M G(x,t;y,s) \ d\mu_{g(s)}(x)  \\
& \geq \frac{C}{(t-s)^{\frac{n}{2}}}\exp\left(-C\frac{{\rm dist}_{g(t)}^2(x,y)}{t-s} - \frac{1}{2\sqrt{t-s}}\int_s^{t}\sqrt{t-\tau}R(x_0,\sigma)d\sigma\right),
\end{align*}
where $C$'s are constants depending on $n,T$ and $g(0)$.
Hence, by our assumptions on the scalar curvature and the diameter, one obtains 
the following theorem, which is essentially due to Q. Zhang in \cite{QZhang12}.

\begin{Prop*UB}[cf. Theorem 1.1 (a) in \cite{QZhang12}]
Let $\{g(t)\}_{t\in [0,T)}$ be the solution of the Ricci flow on a closed Riemannian manifold $(M^n,g(0))$ with $n\geq 3$. If the diameter satisfies ${\rm diam}_{g(t)}\leq c_1\sqrt{T-t}$ and there exist a point $x_0\in M$ such that $R(x_0,t)\leq c_2(T-t)^{-1}$, then we have ${\rm Vol}_{g(t)}\leq C (T-t)^{\frac{n}{2}}$ for all $t>\frac{T}{2}$, where $C$ depends only on $c_1,c_2,n,T$ and $g(0)$. 
\end{Prop*UB}

\begin{proof}
The proof is adapted from Zhang's local volume estimate in \cite{QZhang12}. 
The reader should be careful on tracing the dependence of the constant $C$, which varies line by line, in the following bounds.

Recall that $R(x,t)$ is either nonnegative, or negative somewhere and bounded below by the negative function $(\frac{1}{\min R_{g(0)}}-\frac{2t}{n})^{-1}$ for all $t>0$. 
Moreover, since the manifold is closed and $ \frac{d}{dt} d\mu_{g(t)}= -Rd\mu_{g(t)}$, one can derive 
$$ \frac{d}{dt}\int_M u(x,t)\  d\mu_{g(t)} = \int_M  \Delta_x u(x,t)  -R u(x,t)   d\mu_{g(t)} =  -\int_M R u(x,t)   d\mu_{g(t)}$$
and thus either $\frac{d}{dt}\int_M u(x,t)\  d\mu_{g(t)}\leq 0$ or
$$\frac{d}{dt}\int_M u(x,t)\  d\mu_{g(t)} \leq \frac{n}{2}\left(t-\frac{n}{2\min R_{g(0)}}\right)^{-1} \int_M u(x,t)   d\mu_{g(t)}.$$
Integrating it from $s$ to $t$, we obtain either 
$$\int_M u(x,t)\  d\mu_{g(t)}\leq \lim_{t\searrow s}\int_M u(x,t)\  d\mu_{g(t)}=1$$ 
or
\begin{align*}
\int_M u(x,t)\  d\mu_{g(t)} 
\leq &\ \left(\frac{t-\frac{n}{2\min R_{g(0)}}}{s-\frac{n}{2\min R_{g(0)}}}\right)^{\frac{n}{2}}
\leq \left(1+ C_1(t-s)\right)^{\frac{n}{2}},
\end{align*}
where $C_1=(-\frac{n}{2\min R_{g(0)}})^{-1}$.
Thus the upper bound for the integral heat kernel is obtained. 

We claim that $G(x,t;x_0,s)$ is bounded pointwise from below by $C_2(t-s)^{-\frac{n}{2}}$ for all $t\geq \frac{T+s}{2}$,
where $C_2$ depends on $c_1,c_2,n,T$ and $g(0)$. 
Therefore, combining with the upper bound above,
we have 
$$ \left(1+ C_1(t-s)\right)^{\frac{n}{2}} \geq \int_M G(x,t;x_0,s)\  d\mu_{g(t)}(x)  \geq C_2(t-s)^{-\frac{n}{2}} {\rm Vol}_{g(t)}$$
and thus $${\rm Vol}_{g(t)}\leq   C\left((t-s)+(t-s)^2\right)^{\frac{n}{2}}\leq C(t-s)^{\frac{n}{2}},$$
where the last $C$ depends only on $n,C_1,C_2$ and $T$.
This upper bound holds for all fixed $s$ and all $t\in[\frac{T+s}{2},T]$, so we may choose $t=\frac{T+s}{2}$ and 
derive the conclusion
$${\rm Vol}_{g(t)}\leq C(T-t)^{\frac{n}{2}} \mbox{ for all } t> T/2.$$

Now we complete the proof by verifying the claim that $G(x,t;x_0,s)\geq C_2(t-s)^{-\frac{n}{2}}$.
Note that, fixing $x$ and $t$, $v(y,s):=G(x,t;y,s)$ satisfies the backward conjugate heat equation 
$\partial_s v= -\Delta_y v + Rv$
and thus the function $f(y,s)$ defined by $(4\pi\tau)^{-\frac{n}{2}}e^{-f}=v$ satisfies 
$-f_s=\Delta f -|\nabla f|^2 +R-\frac{n}{2\tau}$, where $\tau=t-s$.
Combining with Perelman's estimate $ \tau (2\Delta f -|\nabla f|^2 +R) +f-n \leq 0$ (cf. \cite[Corollary 9.4]{Perelman02}),
one has $$-f_s \leq \frac{1}{2}R -\frac{1}{2}|\nabla f|^2 -\frac{1}{2\tau}f\leq \frac{1}{2}R -\frac{1}{2\tau}f,\ \mbox{ i.e., }\ -(\sqrt{\tau} f)_s  \leq \frac{1}{2}\sqrt{\tau}R.$$  
Integrating from $s$ to $t$, one has  
$$\sqrt{t-s}f(y,s) \leq   \lim_{\sigma\to t}\sqrt{t-\sigma}f(y,\sigma)  +\frac{1}{2}\int_s^t\sqrt{t-\sigma}R(y,\sigma)d\sigma.$$ 
Because $G(x,t;y,\sigma)$ behaves like $(t-\sigma)^{-\frac{n}{2}}$ as $\sigma\to t$ whenever $x=y$ (cf. \cite[Ch. 24]{ChowPart3}), 
$f(x,\sigma)$ is uniformly bounded as $\sigma\to t$, thus $\lim_{\sigma\to t}\sqrt{t-\sigma}f(x,\sigma)=0$
and $$\sqrt{t-s}f(x,s) \leq  \frac{1}{2}\int_s^t\sqrt{t-\sigma}R(x,\sigma)d\sigma.$$  

Since $x$ can be chosen arbitrarily, one may take $x=x_0$ in the beginning and obtain
$$-f(x_0,s) \geq -\frac{1}{2\sqrt{t-s}}\int_s^t\sqrt{t-\sigma}R(x_0,\sigma)d\sigma\geq -c_2.$$ 
So $$G(x_0,t;x_0,s) =v(x_0,s)  \geq(4\pi(t-s))^{-\frac{n}{2}}e^{-c_2}.$$
Moreover, by gradient estimate of heat equation along the Ricci flow (cf. \cite[(3.44)]{QZhang06}), 
one can compare $v(x,s)$ with $v(x_0,s)$, i.e.,
$$G(x,t;x_0,s)\geq C_3K^{-1} \exp\left(-2C_4\frac{{\rm dist}_{g(t)}^2(x,x_0)}{t-s}\right)(G(x_0,t;x_0,s))^2$$
for all $x\in M$, where $K$ is the upper bound of $G$ and $C_3,C_4$ are universal constants.
In \cite[(1.5)]{QZhang12}, it was proved that $K\leq C_5(t-s)^{-\frac{n}{2}}$. 
Therefore, using ${\rm diam}_{g(t)}\leq c_1\sqrt{T-t}$ and $t\geq\frac{T+s}{2}$, one has
\begin{align*}
G(x,t;x_0,s) \geq &\ C_3K^{-1} \exp\left(-2C_4\frac{{\rm dist}_{g(t)}^2(x,x_0)}{t-s}\right)(G(x_0,t;x_0,s))^2\\
\geq &\ C_3 (4\pi)^{-n}e^{-2c_2}C_5^{-1} \exp\left(-2C_4\frac{2c_1^2(T-t)}{T-s}\right)(t-s)^{-\frac{n}{2}}\\
\geq  &\ C_2(t-s)^{-\frac{n}{2}},\ \mbox{ where }C_2=(4\pi)^{-n}C_3C_5^{-1} e^{-2c_2-4C_4c_1^2}.
\end{align*}
The claim is verified for $C_2$ depending only on $c_1,c_2,n,T$ and $g(0)$. 
\end{proof}

\begin{Rmk}
Suppose that $-\hspace{-3.75mm} \int R(x,t) \ d{\rm vol}_{g(t)}  \geq \frac{n}{2}(T-t)^{-1}$ for all $t>\frac{T}{2}$.
Then 
$$\frac{d}{dt}{\rm Vol}_{g(t)} = -\int_M R\ d{\rm vol}_{g(t)} \leq  -  \frac{n}{2}(T-t)^{-1} {\rm Vol}_{g(t)}$$
implies that ${\rm Vol}_{g(t)} \leq C(T-t)^{\frac{n}{2}}$ for all $t>\frac{T}{2}$, 
where $C=\left(\frac{2}{T}\right)^{\frac{n}{2}}{\rm Vol}_{g(\frac{T}{2})}$.
This might help to remove the curvature assumption in Theorem \ref{UB}.
We remind the reader that, when $-\hspace{-3.75mm} \int R(x,t) \ d{\rm vol}_{g(t)}  \ngeq \frac{n}{2}(T-t)^{-1}$ for some $t_k\to T$, there must exist a sequence of points $(x_k,t_k)$ with Type I blow-up scalar curvature. However, this is insufficient for us to apply Theorem \ref{UB} because we need a fixed point $x_0$.
\end{Rmk}

\begin{Rmk}
For generic Ricci flows, it is not hard to see that volume grows at most polynomially.
Indeed, this is trivial when $R_{g(0)}\geq0$. For $\min R_{g(0)}<0$, by using $\frac{d}{dt} R \geq \Delta R+ \frac{2}{n}R^2$, one can show that $R_{g(t)}\geq \left( \frac{1}{\min R_{g(0)}}-\frac{2t}{n} \right)^{-1}$ and thus 
$\frac{d}{dt} \ln{\rm Vol}_{g(t)}=- -\hspace{-4mm}\int_M R\ d{\rm vol}_{g(t)} \leq  \left( \frac{2t}{n} -\frac{1}{\min R_{g(0)}}\right)^{-1}$.
Hence $${\rm Vol}_{g(t)}\leq {\rm Vol}_{g(0)} \left(-\min R_{g(0)}\right)^{\frac{n}{2}} 
\left( \frac{2t}{n} -\frac{1}{\min R_{g(0)}} \right)^{\frac{n}{2}}$$ for all
$t< T$.
\end{Rmk}


\end{document}